\documentclass{amsart}
\usepackage[utf8]{inputenc}

\usepackage{amssymb}
\usepackage{amsthm}
\usepackage{mathtools}
\usepackage{booktabs}
\usepackage{enumitem}
\usepackage{tikz-cd}

\newtheorem{thm}{Theorem}[section]
\newtheorem{prop}[thm]{Proposition}
\newtheorem{lem}[thm]{Lemma}

\theoremstyle{definition}

\newtheorem{question}[thm]{Question}

\theoremstyle{remark}

\numberwithin{equation}{section}
\makeatletter
\def\@setaddresses{\par
  \nobreak \begingroup
  \footnotesize
  \def\author##1{\nobreak\addvspace\bigskipamount}%
  \def\\{\unskip, \ignorespaces}%
  \interlinepenalty\@M
  \def\address##1##2{\begingroup
    \par\addvspace\bigskipamount\indent
    \@ifnotempty{##1}{(\ignorespaces##1\unskip) }%
    {\scshape\ignorespaces##2}\par\endgroup}%
  \def\email##1##2{\begingroup
    \@ifnotempty{##2}{\nobreak\indent
      \@ifnotempty{##1}{, \ignorespaces##1\unskip}
      \unskip\ttfamily##2\par}\endgroup}%
  \addresses
  \endgroup}
\makeatother

\title{On O'Grady's generalized Franchetta conjecture}

\author{Nebojsa Pavic}
\author{Junliang Shen}
\author{Qizheng Yin}

\address{Departement Mathematik\\ETH Z\"urich\\R\"amistrasse 101\\8092 Z\"urich\\Switzerland}

\email{npavic@student.ethz.ch, junliang.shen@math.ethz.ch, qizheng.yin@math.ethz.ch}

\date{\today}

\begin{document}

\begin{abstract}
We study relative zero cycles on the universal polarized $K3$ surface $X \to \mathcal{F}_g$ of degree $2g - 2$. It was asked by O'Grady if the restriction of any class in $\mathrm{CH}^2(X)$ to a closed fiber $X_s$ is a multiple of the Beauville--Voisin canonical class $c_{X_s} \in \mathrm{CH}_0(X_s)$. Using Mukai models, we give an affirmative answer to this question for $g \leq 10$ and $g = 12, 13, 16, 18, 20$.
\end{abstract}

\maketitle

\section{Introduction}
\noindent Throughout, we work over the complex numbers. Let $S$ be a projective $K3$ surface. In \cite{BV}, Beauville and Voisin studied the Chow ring $\mathrm{CH}^\ast(S)$ of $S$. They showed that there is a canonical class $c_{S} \in \mathrm{CH}_0(S)$ represented by a point on a rational curve in $S$, which satisfies the following properties:
\begin{enumerate}[labelindent=\parindent,leftmargin=*]
\item[(i)] The intersection of two divisor classes on $S$ always lies in $\mathbb{Z}c_S \subset \mathrm{CH}_0(S)$.
\item[(ii)] The second Chern class $c_2(T_S)$ equals $24 c_S \in \mathrm{CH}_0(S)$. 
\end{enumerate}
This result is rather surprising since the Chow group $\mathrm{CH}_0(S)$ is infinite-dimensional by Mumford's theorem \cite{Mum}.

Let $\mathcal{F}_g$ denote the moduli space of (primitively) polarized $K3$ surfaces of degree $2g - 2$. For $g \geq 3$, let $\mathcal{F}_g^0 \subset \mathcal{F}_g$ be the open dense subset parametrizing polarized~$K3$ surfaces with trivial automorphism groups, which carries a universal family $X \rightarrow \mathcal{F}_g^0$. Motivated by Franchetta’s conjecture on the moduli spaces of curves (see~\cite{AC}), O'Grady asked the following question in \cite{OG}, referred to as the generalized Franchetta conjecture.

\begin{question}[Generalized Franchetta conjecture]\label{Conj}
Given a class $\alpha \in \mathrm{CH}^2(X)$ and a closed point $s \in \mathcal{F}_g^0$, is it true that $\alpha|_{X_s} \in \mathbb{Z}c_{X_s}$?
\end{question}

The goal of this paper is to give an affirmative answer to Question \ref{Conj} for a list of small values of $g$. By the work of Mukai \cite{Muk1, Muk3, Muk4, Muk5}, for these $g$ a general polarized $K3$ surface can be realized in a variety with ``small'' Chow groups as a complete intersection with respect to a vector bundle.

\begin{thm}\label{Main}
The generalized Franchetta conjecture holds for $g\leq 10$ and $g=12$, $13, 16, 18, 20$.
\end{thm}

The paper is organized as follows. In Section 1 we review Mukai's constructions and make some comments about Question \ref{Conj}. In Section 2 we prove Theorem \ref{Main} for all cases except $g = 13, 16$. Two independent proofs are presented, one using Voisin's result~\cite{V2}, the other via a direct calculation. The cases $g = 13, 16$ have a different flavor and are treated in Section 3.

This work is inspired by a recent preprint of Pedrini \cite{Ped}. However, contrary to what was claimed there, it does not suffice to show that $\mathrm{CH}^2(X)_{\mathbb{Q}}$ is finite-dimensional. Our proof relies deeply on the result of Beauville--Voisin \cite{BV}.

\subsection*{Acknowledgement}
We are grateful to Rahul Pandharipande for his constant support and his enthusiasm in this project. We also thank Kieran O'Grady for his careful reading of a preliminary version of this paper. J.~S.~and Q.~Y.~were supported by the grant ERC-2012-AdG-320368-MCSK.

\section{Mukai models and the basic setting}

\noindent In this section we review Mukai's work \cite{Muk1, Muk3, Muk4, Muk5} on the projective models of general polarized $K3$ surfaces of small degrees. Using Mukai's models, we set up the framework for the proof of Theorem \ref{Main}.

The following table summarizes the ambient varieties $\mathbb{G}_g$ and  vector bundles $\mathcal{U}_g$ involved in the constructions. It is also accompanied by a glossary.
\[\arraycolsep=8pt
\begin{array}{lll|lll}
\toprule
g & \mathbb{G}_g & \mathcal{U}_g & g & \mathbb{G}_g & \mathcal{U}_g \\
\midrule
2 & \mathbb{P}(1, 1, 1, 2) & \mathcal{O}(6) & 9 & \mathbb{G}(3, 6) & \mathcal{O}(1)^{\oplus 4} \oplus \wedge^2\mathcal{Q} \\
3 & \mathbb{P}^3 & \mathcal{O}(4) & 10 & \mathbb{G}(2, 7) & \mathcal{O}(1)^{\oplus 3} \oplus \wedge^4\mathcal{Q} \\
4 & \mathbb{P}^4 & \mathcal{O}(2) \oplus \mathcal{O}(3) & 12 & \mathbb{G}(3, 7) & \mathcal{O}(1) \oplus (\wedge^2\mathcal{E}^\vee)^{\oplus 3} \\
5 & \mathbb{P}^5 & \mathcal{O}(2)^{\oplus 3} & 13 & \mathbb{G}(3, 7) & (\wedge^2\mathcal{E}^\vee)^{\oplus 2} \oplus \wedge^3\mathcal{Q} \\
6 & \mathbb{G}(2, 5) & \mathcal{O}(1)^{\oplus 3} \oplus \mathcal{O}(2) & 16 & \mathbb{G}(2, 3, 4) & \mathcal{V}_{16}^{\oplus 2} \oplus \widetilde{\mathcal{V}}_{16}^{\oplus 2} \\
7 & \mathbb{OG}(5, 10) & \mathcal{V}_7^{\oplus 8} & 18 & \mathbb{OG}(3, 9) & \mathcal{V}_{18}^{\oplus 5} \\
8 & \mathbb{G}(2, 6) & \mathcal{O}(1)^{\oplus 6} & 20 & \mathbb{G}(4, 9) & (\wedge^2\mathcal{E}^\vee)^{\oplus 3} \\
\bottomrule
\end{array}\]
\begin{itemize}[align=left,labelindent=\parindent,labelwidth=\widthof{$\mathbb{P}(1, 1, 1, 2)$},labelsep=1.5em,leftmargin=!]
\item[$\mathbb{P}(1, 1, 1, 2)$] $3$-dimensional weighted projective space with weights $(1, 1, 1, 2)$
\item[$\mathbb{G}(r, n)$] Grassmannian parametrizing $r$-dimensional subspaces of a fixed $n$-dimensional vector space
\item[$\mathcal{O}(i)$] line bundle on $\mathbb{G}(r, n)$ with respect to the Pl\"ucker embedding
\item[$\mathbb{OG}(r, n)$] orthogonal Grassmannian parametrizing $r$-dimensional isotropic subspaces of a fixed $n$-dimensional vector space equipped with a nondegenerate symmetric $2$-form
\item[$\mathcal{V}_7$] line bundle on $\mathbb{OG}(5, 10)$ corresponding to a spin representation
\item[$\mathcal{Q}$] universal quotient bundle on $\mathbb{G}(r, n)$
\item[$\mathcal{E}$] universal subbundle on $\mathbb{G}(r, n)$
\item[$\mathbb{G}(2, 3, 4)$] Ellingsrud--Piene--Str{\o}mme moduli space of twisted cubic curves, constructed as the GIT quotient of $\mathbb{C}^2\otimes \mathbb{C}^3\otimes \mathbb{C}^4$ by the action of $\mathrm{GL}_2\times \mathrm{GL}_3$ (see \cite{EPS})
\item[$\mathcal{V}_{16}$] rank $3$ tautological vector bundle on $\mathbb{G}(2, 3, 4)$
\item[$\widetilde{\mathcal{V}}_{16}$] rank $2$ tautological vector bundle on $\mathbb{G}(2, 3, 4)$
\item[$\mathcal{V}_{18}$] rank $2$ vector bundle on $\mathbb{OG}(3, 9)$ corresponding to a spin representation
\end{itemize}

For all $g$ listed above, Mukai showed that a general $K3$ surface over $\mathcal{F}_g$ is given as the zero locus of a general global section of $\mathcal{U}_g$ (the cases $g \leq 5$ are classical).

Let
\[\mathbb{P}_g = \mathbb{P}H^0(\mathbb{G}_g, \mathcal{U}_g)\]
be the projectivization of the space of global sections of~$\mathcal{U}_g$, and let
\[Y = \{(s, x) \in \mathbb{P}_g \times \mathbb{G}_g \,|\, s(x) = 0\}\]
be the incidence scheme. We have a diagram
\[\begin{tikzcd}
Y \arrow{r}{\iota} \arrow{d}{\pi} & \mathbb{G}_g \\
\mathbb{P}_g
\end{tikzcd}\]
where $\pi, \iota$ are the two projections.

The discussion above shows that a general fiber of $\pi: Y \to \mathbb{P}_g$ is a polarized $K3$ surface of degree $2g - 2$, and that $\mathbb{P}_g$ rationally dominates the moduli space $\mathcal{F}_g$. Moreover, since $\mathcal{U}_g$ is globally generated, we know that $\iota: Y \to \mathbb{G}_g$ is a projective bundle. Its fiber over a point $x \in \mathbb{G}_g$ is given by
\[\mathbb{P}H^0(\mathbb{G}_g, \mathcal{U}_g \otimes \mathcal{I}_x),\]
where $\mathcal{I}_x$ is the ideal sheaf of $x$. We have the following lemma regarding the Chow group $\mathrm{CH}^2(Y)$ and its restriction to a general fiber of $\pi$.

\begin{lem}\label{Lem0}
Given a closed point $s \in \mathbb{P}_g$ with $K3$ fiber $Y_s$, let $\phi_s: Y_s \hookrightarrow Y$ and $\iota_s: Y_s \hookrightarrow \mathbb{G}_g$ be the natural embeddings. Then we have
\[\mathrm{Im}(\phi_s^*: \mathrm{CH}^2(Y)_{\mathbb{Q}} \to \mathrm{CH}_0(Y_s)_{\mathbb{Q}}) = \mathrm{Im}(\iota_s^*: \mathrm{CH}^2(\mathbb{G}_g)_{\mathbb{Q}} \to \mathrm{CH}_0(Y_s)_{\mathbb{Q}}).\]
\end{lem}

\begin{proof}
Let $\xi \in \mathrm{CH}^1(Y)$ be the relative hyperplane class of $\iota: Y \to \mathbb{G}_g$. By the projective bundle formula, we have for $k \geq 0$,
\begin{equation} \label{Proj}
\mathrm{CH}^k(Y) = \xi^k \cdot \iota^*\mathrm{CH}^0(\mathbb{G}_g) \oplus \xi^{k - 1} \cdot \iota^*\mathrm{CH}^1(\mathbb{G}_g) \oplus \cdots \oplus \iota^*\mathrm{CH}^k(\mathbb{G}_g).
\end{equation}
Let $h \in \mathrm{CH}^1(\mathbb{P}_g)$ be the hyperplane class. Then we have
\[\pi^*h = a \cdot \xi + \iota^*\beta\]
for some $a \in \mathbb{Z}$ and $\beta \in \mathrm{CH}^1(\mathbb{G}_g)$. We claim that $a \neq 0$, otherwise \[\pi^*(h^{\dim \mathbb{P}_g}) = \iota^*(\beta^{\dim \mathbb{P}_g}).\]
Since $\dim \mathbb{P}_g > \dim \mathbb{G}_g$, the right-hand side vanishes, but the left-hand side is the pullback of a point class and is nonzero. Contradiction. Hence
\[\xi = \frac{1}{a}(\pi^*h - \iota^*\beta) \in \mathrm{CH}^1(Y)_{\mathbb{Q}}.\]
The lemma then follows from \eqref{Proj} for $k = 2$ and the fact that $\phi_s^*\pi^*h = 0$.
\end{proof}

We end this section by a few remarks on the generalized Franchetta conjecture.
\begin{enumerate}[labelindent=\parindent,leftmargin=*]
\item[(i)] By a standard ``spreading out'' argument (see \cite[Chapter 1]{V1}), it is equivalent to answer Question \ref{Conj} for general (in fact, very general) fibers $X_s$ over~$\mathcal{F}_g^0$. Moreover, classes in $\mathrm{CH}^2(X)$ supported over a proper closed subset of $\mathcal{F}_g^0$ vanish when restricted to a fiber $X_s$. 

Hence one may work with a family $Y \to B$ such that a general fiber $Y_s$ is a polarized $K3$ surface of degree $2g - 2$, and that $B$ rationally dominates~$\mathcal{F}_g$ via the natural rational map $B \dashrightarrow \mathcal{F}_g$. It then suffices to answer (the analog of) Question \ref{Conj} for classes in $\mathrm{CH}^2(Y)$ and $K3$ fibers $Y_s$. See Section 3 for an even more precise statement.

One may also formulate Question \ref{Conj} in terms of the Chow group $\mathrm{CH}_0(X_\eta)$ of the generic fiber $X_\eta$, but we omit this point of view.

\item[(ii)] By Ro{\u\i}tman's theorem \cite{Ro}, the Chow group $\mathrm{CH}_0(S)$ of a complex $K3$ surface~$S$ is torsion-free. Hence in Question \ref{Conj} it is equivalent to work with $\mathbb{Q}$-coefficients. This also means that under Lemma \ref{Lem0}, we have
\[\mathrm{Im}(\phi_s^*: \mathrm{CH}^2(Y) \to \mathrm{CH}_0(Y_s)) \subset \mathbb{Z}c_{Y_s}\]
if and only if
\[\mathrm{Im}(\iota_s^*: \mathrm{CH}^2(\mathbb{G}_g) \to \mathrm{CH}_0(Y_s)) \subset \mathbb{Z}c_{Y_s}.\]

\item[(iii)] Instead of restricting to $\mathcal{F}_g^0$, one may work with the moduli stack and the universal family over it. Question \ref{Conj} can then be formulated using the Chow groups of smooth Deligne-Mumford stacks with $\mathbb{Q}$-coefficients. This notably covers the case $g = 2$, where a general $K3$ surface over $\mathcal{F}_2$ carries an involution. Our proof in Section 2 works in this case without change.
\end{enumerate}

\section{Polarized $K3$ surfaces as unique complete intersections}
\noindent In this section we deal with the cases $g \leq 10$ and $g=12, 18, 20$. For these $g$, the Mukai model embeds a general polarized $K3$ surface of degree $2g - 2$ in $\mathbb{G}_g$ as a complete intersection with respect to $\mathcal{U}_g$, and the embedding is unique up to automorphisms of $\mathbb{G}_g$ and~$\mathcal{U}_g$. Moreover, the variety $\mathbb{G}_g$ is a Grassmannian or an orthogonal Grassmannian.

Since $\mathbb{P}_g$ rationally dominates the moduli space $\mathcal{F}_g$, to prove Theorem \ref{Main} it suffices to show that the restriction of any class in $\mathrm{CH}^2(Y)$ to a $K3$ fiber $Y_s$ lies in~$\mathbb{Z}c_{Y_s}$. By Lemma \ref{Lem0}, it is equivalent to show that 
\[
\mathrm{Im}(\iota_s^*: \mathrm{CH}^2(\mathbb{G}_g) \to \mathrm{CH}^2(Y_s)) \subset \mathbb{Z}c_{Y_s}.
\]
This allows us to work with a single $K3$ surface $S$ with an embedding
\[i: S \hookrightarrow \mathbb{G}_g.\]

If $g \leq 5$, the variety $\mathbb{G}_g$ is a projective space and its Chow ring is generated by the hyperplane class. Thus Theorem \ref{Main} follows from property (i) of $c_S$ in Section~0.

Now assume that ${\mathbb{G}_g}$ is not a projective space. It is well-known that the Chow group $\mathrm{CH}^2({\mathbb{G}(r,n)})$ of the Grassmannian is generated by the Chern classes $c_1(\mathcal{Q})^2$ and $c_2(\mathcal{Q})$, where $\mathcal{Q}$ is the universal quotient bundle. For the orthogonal Grassmannians, we have instead
\[
\mathrm{CH}^2(\mathbb{OG}(5,10))=\mathbb{Z} \Big(\frac{1}{2}c_2(\mathcal{Q})\Big) \oplus \mathbb{Z} \Big(\frac{1}{4}c_1(\mathcal{Q})^2\Big)
\]
and
\[
\mathrm{CH}^2(\mathbb{OG}(3,9))= \mathbb{Z} \Big(\frac{1}{2}c_2(\mathcal{Q})\Big) \oplus \mathbb{Z} c_1(\mathcal{Q})^2,
\]
where $\mathcal{Q}$ is the corresponding universal quotient bundle (see \cite{Ta}). Hence in all cases a class $\alpha \in \mathrm{CH}^2(\mathbb{G}_g)$ can be uniquely expressed as
\[
\alpha = a \cdot c_2(\mathcal{Q}) + b \cdot c_1(\mathcal{Q})^2,
\]
with $a \in \mathbb{Z}$ if $\mathbb{G}_g$ is a Grassmannian, or $a \in \frac{1}{2} \mathbb{Z}$ if $\mathbb{G}_g$ is an orthogonal Grassmannian. For convenience we define the index $I(\alpha)$ of $\alpha \in \mathrm{CH}^2(\mathbb{G}_g)$ to be the coefficient $a$.

By property (i) of $c_S$ in Section 0, we have $i^\ast (c_1(\mathcal{Q})^2) \in \mathbb{Z}c_S$. Hence the following proposition implies Theorem \ref{Main} for $g=6,7,8,9,10,12,18,20$.

\begin{prop}\label{Prop}
With the notation as above, we have $i^\ast c_2(\mathcal{Q}) \in \mathbb{Z}c_S$.
\end{prop}

We give two independent proofs of the proposition.
\begin{proof}[First proof]
Mukai showed in \cite{Muk1,Muk3} that the restriction of either $\mathcal{Q}$ or $\mathcal{E}^\vee$ to a general $S$ is simple and rigid, where $\mathcal{E}$ is the universal subbundle. In fact, the rigidity ensures that the embedding of $S$ in $\mathbb{G}_g$ is unique. The proposition follows from a strong result of Voisin \cite[Corollary 1.10]{V2} that the second Chern class of any simple rigid vector bundle on a $K3$ surface $S$ lies in $\mathbb{Z}c_S$, which was conjectured by Huybrechts earlier in \cite{Huy}.
\end{proof}

Since part of the original motivation of the generalized Franchetta conjecture was to make Huybrechts' conjecture as its consequence (see \cite[Section 5]{OG}), we give a direct proof of Proposition \ref{Prop} without using Voisin's result.

\begin{proof}[Second proof]
We first consider the cases where $\mathbb{G}_g$ is a Grassmannian. The standard exact sequence of normal bundles
\[
0 \rightarrow T_S \rightarrow i^\ast T_{\mathbb{G}_g} \rightarrow i^\ast \mathcal{U}_g \rightarrow 0
\]
yields the following relation in $\mathrm{CH}_0(S)$:
\begin{equation}\label{relation}
i^\ast c_2(T_{\mathbb{G}_g}) = c_2(T_S) + i^\ast c_2(\mathcal{U}_g).
\end{equation}
Here $T_{\mathbb{G}_g}$ and $T_S$ are the corresponding tangent bundles. Using the index of classes in $\mathrm{CH}^2(\mathbb{G}_g)$, the relation \eqref{relation} can be written as 
\begin{equation}\label{relation0}
\Big{(}I(c_2(T_{\mathbb{G}_g}))-I(c_2(\mathcal{U}_g)) \Big{)} \cdot i^\ast c_2(\mathcal{Q}) = c_2(T_S) + \gamma,
\end{equation}
where $\gamma$ can be expressed in terms of divisor classes on $S$. By properties (i) and~(ii) of $c_S$ in Section 0, both $c_2(T_S)$ and $\gamma$ lie in $\mathbb{Z}c_S$. Hence it suffices to verify that 
\begin{equation}\label{Lemma}
I(c_2(T_{\mathbb{G}_g}))-I(c_2(\mathcal{U}_g)) \neq 0.
\end{equation}
The tangent bundle $T_{\mathbb{G}(r, n)}$ of the Grassmannian is $\mathcal{H}om(\mathcal{E},\mathcal{Q})$, where $\mathcal{E}$ is the universal subbundle. By computing the Chern character
\[
\mathrm{ch}(\mathcal{H}om(\mathcal{E},\mathcal{Q})) = \mathrm{ch}(\mathcal{E}^\vee \otimes \mathcal{Q}) = \mathrm{ch}(\mathcal{E}^\vee) \cdot \mathrm{ch}(\mathcal{Q})
\]
and the standard relation $c(\mathcal{E})\cdot c(\mathcal{Q})=1$ between the total Chern classes, we have
\[
I(c_2(T_{\mathbb{G}(r,n)})) = 2r-n.
\]
The following is a case-by-case study:
\begin{itemize}[align=left,labelindent=\parindent,labelwidth=\widthof{$g=12, 20$},labelsep=1.5em,leftmargin=!]
\item[$g=6, 8$] Here $\mathcal{U}_g$ is a direct sum of line bundles. Hence $I(c_2(\mathcal{U}_g)) = 0$ and
\[
I(c_2(T_{\mathbb{G}_g}))-I(c_2(\mathcal{U}_g)) = 2r-n \neq 0.
\]
\item[$g=9$] We have $I(c_2(T_{\mathbb{G}_9})) = 0$ and $I(c_2(\mathcal{U}_9)) = I(c_2(\wedge^2\mathcal{Q})) = 1$. Hence
\[
I(c_2(T_{\mathbb{G}_9}))-I(c_2(\mathcal{U}_9)) = -1 \neq 0.
\]
\item[$g=10$] We have $I(c_2(T_{\mathbb{G}_{10}})) = -3$ and $I(c_2(\mathcal{U}_{10})) = I(c_2(\wedge^4\mathcal{Q})) = 1$.~Hence
\[
I(c_2(T_{\mathbb{G}_{10}}))-I(c_2(\mathcal{U}_{10})) = -4 \neq 0.
\]
\item[$g=12, 20$] We have $I(c_2(T_{\mathbb{G}_g})) = -1$ and $I(c_2(\mathcal{U}_g)) = 3 I(c_2(\wedge^2\mathcal{E}^\vee))$. Hence
\[
I(c_2(T_{\mathbb{G}_g}))-I(c_2(\mathcal{U}_g)) = -1 - 3 I(c_2(\wedge^2\mathcal{E}^\vee)) \neq 0. 
\]
\end{itemize}

The orthogonal Grassmannian cases ($g=7, 18$) are similar. The relation \eqref{relation0} still holds, and it suffices to show \eqref{Lemma}. Note that
the left-hand side of \eqref{Lemma} may be a half integer.

The natural embedding $j: \mathbb{OG}(r,n) \hookrightarrow \mathbb{G}(r,n)$ can be realized as the zero locus of a smooth section of the vector bundle $\mathcal{W}$, which is given by the cohomology group $H^0(\mathbb{P}^{r-1}, \mathcal{O}(2))$ over every closed point \[ [\mathbb{P}^{r-1}\subset \mathbb{P}^{n-1}] \in \mathbb{G}(r,n).\]
Hence we have \[I(c_2(T_{\mathbb{OG}(r,n)})) = I(j^\ast c_2(T_{\mathbb{G}(r,n)})) - I(j^\ast c_2(\mathcal{W})).\] The term $I(j^\ast c_2(T_{\mathbb{G}(r,n)}))$ was already calculated, and the term $I(j^\ast c_2(\mathcal{W}))$ can be determined by the following Grothendieck--Riemann--Roch calculation.

We consider $p: \mathbb{P}(\mathcal{E})\rightarrow \mathbb{G}(r,n)$ the projective bundle on $\mathbb{G}(r,n)$ associated to the universal subbundle $\mathcal{E}$. Let $L=\mathcal{O}_{\mathbb{P}(\mathcal{E})}(1)$ and let $\xi$ be the relative hyperplane class $c_1(L)$. We have $R^kp_\ast L= 0$ for $k>0$. Hence by the Grothendieck--Riemann--Roch theorem, we have
\[
\mathrm{ch}(\mathcal{W}) = \mathrm{ch}(Rp_\ast L^{\otimes 2})
= p_\ast (\mathrm{exp}(2\xi)\cdot \mathrm{td}(T_p)).
\]
Together with the exact sequence
\[
0 \rightarrow \mathcal{O}_{\mathbb{P}(\mathcal{E})} \rightarrow p^\ast \mathcal{E} \otimes L \rightarrow T_p \rightarrow 0,
\]
we obtain for $r = 3, 5$,
\[
I(c_2(\mathcal{W})) = -(r + 2).
\]

We finish the proof of Proposition \ref{Prop}:
\begin{itemize}[align=left,labelindent=\parindent,labelwidth=\widthof{$g=18$},labelsep=1.5em,leftmargin=!]
\item[$g=7$] Here $\mathcal{U}_7$ is a direct sum of line bundles. Hence $I(c_2(\mathcal{U}_7))=0$ and
\[
I(c_2(T_{\mathbb{G}_7})) - I(c_2(\mathcal{U}_7)) = 0 - (-7) = 7 \neq 0.
\]
\item[$g=18$] We have $I(c_2(T_{\mathbb{G}_{18}})) = -3 - (-5)$ and $I(c_2(\mathcal{U}_{18})) = 5 I(c_2(\mathcal{V}_{18}))$. Hence
\[
I(c_2(T_{\mathbb{G}_{18}})) - I(c_2(\mathcal{U}_{18}))  = 2 - 5 I(c_2(\mathcal{V}_{18})) \neq 0. \qedhere
\]
\end{itemize}
\end{proof}

\section{Polarized $K3$ surfaces as nonunique complete intersections}
\noindent In this section we treat the remaining cases $g = 13, 16$. In both cases, the embedding of a polarized $K3$ surface $S$ of degree $2g - 2$ in $\mathbb{G}_g$ is not unique and the restriction of the tautological bundles to $S$ may not be rigid. Hence the methods in Section 2 break down.

We keep the notation of Section 1 and write $\Phi: \mathbb{P}_g \dashrightarrow \mathcal{F}_g$ for the dominant rational map. Let $t \in \mathcal{F}_g^0$ be a closed point outside the indeterminacy locus of~$\Phi$ in~$\mathcal{F}_g$. Given two closed points $s_1, s_2 \in \mathbb{P}_g$ with $\Phi(s_1) = \Phi(s_2) = t$, there are canonical isomorphisms
\[Y_{s_1} \cong Y_{s_2} \cong X_t.\]
We identify $\mathrm{CH}_0(Y_{s_1}),  \mathrm{CH}_0(Y_{s_2})$ with $\mathrm{CH}_0(X_t)$, and define
\[\mathrm{CH}^2(Y)_{\mathrm{inv}} = \{\alpha \in \mathrm{CH}^2(Y) \,|\, \phi_{s_1}^\ast \alpha = \phi_{s_2}^\ast \alpha \textrm{ for all } s_1, s_2 \in \mathbb{P}_{g} \textrm{ above}\}.\]
Recall that $\phi_s: Y_s \hookrightarrow Y$ is the natural embedding for $s \in \mathbb{P}_g$.

Again by the ``spreading out'' argument and the fact that classes supported over a proper closed subset of $\mathcal{F}_g^0$ do not contribute, to prove Theorem \ref{Main} it suffices show that
\[
\mathrm{Im}(\phi_s^*: \mathrm{CH}^2(Y)_{\mathrm{inv}} \to \mathrm{CH}_0(Y_s)) \subset \mathbb{Z}c_{Y_s}
\]
for all (or general, or very general) $K3$ fibers $Y_s$.

First we consider the case $g=13$. As described by the Mukai model, let
\[
i: S \hookrightarrow \mathbb{G}(3,7).
\]
be the embedding of a $K3$ surface $S$ in $\mathbb{G}(3,7)$. The restriction of $\mathcal{E}^\vee$ (dual of the universal subbundle) to $S$ is semi-rigid, which carries a $2$-dimensional deformation. Let $M_S$ be the moduli space of stable vector bundles on $S$ with Mukai vector $(3, H, 4)$, where $H$ is the polarization class. A general point of $M_S$ is represented by $i^*\mathcal{E}^\vee$ for some $i$; see \cite{Muk4} for details. Note that $M_S$ is also a polarized $K3$ surface with $g = 13$.

Let $s \in \mathbb{P}_{13}$ be a closed point with $K3$ fiber $Y_s$, and let $\iota_s: Y_s \hookrightarrow \mathbb{G}(3, 7)$ be as in Section 1. By Lemma \ref{Lem0}, the restriction $\phi_s^*\alpha$ of a class $\alpha \in \mathrm{CH}^2(Y)_{\mathrm{inv}}$ can be expressed as
\begin{equation}\label{express}
\phi_s^*\alpha = a \cdot \iota_s^\ast c_2(\mathcal{Q}) + b \cdot \iota_s^\ast (c_1(\mathcal{Q})^2),
\end{equation}
where $\mathcal{Q}$ is the universal quotient bundle and $a, b \in \mathbb{Q}$ are constants independent of $s \in \mathbb{P}_{13}$. By property (i) of $c_{Y_s}$ in Section 0, we have $\iota_s^\ast (c_1(\mathcal{Q})^2) \in \mathbb{Z}c_{Y_s}$.

Theorem \ref{Main} for $g=13$ is a direct consequence of the following lemma.

\begin{lem}\label{lastLem}
In the expression \eqref{express}, the coefficient $a$ is zero.
\end{lem}
\begin{proof}
We choose closed points $s_1, s_2 \in \mathbb{P}_{13}$ with $\Phi(s_1) =\Phi(s_2) = t \in \mathcal{F}_{13}^0$, such that the vector bundles $\iota_{s_1}^\ast \mathcal{E}^\vee, \iota_{s_2}^\ast \mathcal{E}^\vee$ represent different point classes in $\mathrm{CH}_0(M_{X_t})$. This is possible by \cite[Theorem 2]{Muk4}, which shows that $\mathbb{P}_{13}$ rationally dominates the moduli space of triples $(S, H, E)$ where $S$ is a $K3$ surface, $H$ is a polarization with $H^2 = 24$, and $E$ is a stable vector bundle with Mukai vector $(3,H,4)$. 

Since $\alpha \in \mathrm{CH}^2(Y)_{\mathrm{inv}}$, we have by definition $\phi_{s_1}^\ast \alpha = \phi_{s_2}^\ast \alpha$ and hence
\begin{equation}\label{AAA}
a \cdot \iota_{s_1}^\ast c_2(\mathcal{Q}) = a \cdot \iota_{s_2}^\ast c_2(\mathcal{Q}),
\end{equation}
viewed as an equality in $\mathrm{CH}_0(X_t)_{\mathbb{Q}}$.

On the other hand, let $\mathbb{F}$ be a universal sheaf over $M_{X_t} \times X_t$ (which exists by the numerics of the Mukai vector; see \cite[Corollary 4.6.7]{HL}). The correspondence \[
\mathrm{ch}(\mathbb{F}) \cdot \sqrt{\mathrm{td}(T_{M_{X_t} \times X_t})} \in \mathrm{CH}^*(M_{X_t} \times X_t)_{\mathbb{Q}}\]
induces an isomorphism of (ungraded) Chow groups
\[
\theta: \mathrm{CH}^\ast (M_{X_t}) \xrightarrow{\simeq} \mathrm{CH}^\ast (X_t).
\]
Here for $[E] \in M_{X_t}$, we have
\[\theta([E]) = 3[X_t] + H + 15c_{X_t} - c_2(E) \in \mathrm{CH}^*(X_t).\]

According to our choice of $s_1,s_2 \in \mathbb{P}_{13}$, the vector bundles $\iota_{s_1}^\ast \mathcal{E}^\vee, \iota_{s_2}^\ast \mathcal{E}^\vee$ represent different classes in $\mathrm{CH}_0(M_{X_t})$. By applying $\theta$, we find
\[\iota_{s_1}^\ast c_2(\mathcal{E}^\vee) \neq \iota_{s_2}^\ast c_2(\mathcal{E}^\vee)\]
in $\mathrm{CH}_0(X_t)$, and together with \eqref{AAA} we obtain $a=0$.
\end{proof}

Finally we consider the case $g=16$. The variety $\mathbb{G}_{16}=\mathbb{G}(2,3,4)$ is realized as a GIT quotient of $\mathbb{C}^2\otimes \mathbb{C}^3 \otimes \mathbb{C}^4$ by the obvious action of $\mathrm{GL}_2 \times \mathrm{GL}_3$ on the first two factors. As described in \cite{EPS} (see also \cite{Muk5}), there are two tautological vector bundles $\mathcal{V}_{16}$ and~$\widetilde{\mathcal{V}}_{16}$ of rank 3 and 2 respectively, as well as a morphism 
\[
\mathcal{V}_{16}\otimes (\mathbb{C}^4)^\vee \rightarrow \widetilde{\mathcal{V}}_{16}.
\]
Further, it was shown in \cite[Proposition 2]{EPS} that the Chow ring $\mathrm{CH}^\ast(\mathbb{G}(2,3,4))$ is generated by the Chern classes of $\mathcal{V}_{16}, \widetilde{\mathcal{V}}_{16}$. To prove Theorem \ref{Main} we have to take care of the second Chern classes of both tautological bundles.

Let $i: S\hookrightarrow \mathbb{G}(2,3,4)$ be the embedding of a $K3$ surface $S$ in $\mathbb{G}(2,3,4)$ as in the Mukai model. By the same reasoning as in Section 2, we have the following relation in $\mathrm{CH}_0(S)$:
\begin{equation}\label{lastrelation}
i^*c_2(T_{\mathbb{G}(2,3,4)}) = c_2(T_S)+i^*c_2(\mathcal{U}_{16}).
\end{equation}
Here $\mathcal{U}_{16}= \mathcal{V}_{16}^{\oplus 2} \oplus \widetilde{\mathcal{V}}_{16}^{\oplus 2}$. Using the exact sequence (see \cite[(4-4)]{ES})
\[
0 \to \mathcal{O}_{\mathbb{G}(2,3,4)} \to \mathcal{E}nd(\mathcal{V}_{16}) \oplus \mathcal{E}nd(\widetilde{\mathcal{V}}_{16}) \to \mathcal{H}om(\mathcal{V}_{16}\otimes (\mathbb{C}^4)^\vee, \widetilde{\mathcal{V}}_{16}) \to T_{\mathbb{G}(2,3,4)} \to 0,
\]
the relation \eqref{lastrelation} can be written as
\[
6c_2(i^\ast \widetilde{\mathcal{V}}_{16}) = c_2(T_S)+\gamma,
\]
where $\gamma$ can be expressed in terms of divisor classes on $S$. By properties (i) and~(ii) of $c_S$ in Section $0$, this verifies that $i^\ast c_2(\widetilde{\mathcal{V}}_{16}) \in \mathbb{Z}c_S$. 

Alternatively, by Mukai's results \cite[Propositions 1.3 and 2.2]{Muk5}, for a general~$S$ the vector bundle $i^\ast \widetilde{\mathcal{V}}_{16}$ is simple and rigid. The statement $i^\ast c_2(\widetilde{\mathcal{V}}_{16}) \in \mathbb{Z}c_S$ also follows from Voisin's result \cite[Corollary 1.10]{V2}.

Let $s \in \mathbb{P}_{16}$ be a closed point with $K3$ fiber $Y_s$, and let $\iota_s: Y_s \hookrightarrow \mathbb{G}(2, 3, 4)$ be as before. Again by Lemma \ref{Lem0} and property (i) of $c_{Y_s}$ in Section 0, the restriction $\phi_s^*\alpha$ of a class $\alpha \in \mathrm{CH}^2(Y)_{\mathrm{inv}}$ can be expressed as 
\[
\phi_s^*\alpha = a\cdot \iota_s^\ast c_2(\mathcal{V}_{16}) + \tilde{a} \cdot\iota_s^\ast c_2(\widetilde{\mathcal{V}}_{16}) + b \cdot c_{Y_s},
\]
where $a, \tilde{a}, b \in \mathbb{Q}$ are constants independent of $s \in \mathbb{P}_{16}$. Moreover, the fact that $\iota_s^\ast c_2(\widetilde{\mathcal{V}}_{16}) \in \mathbb{Z}c_{Y_s}$ implies
\begin{equation}\label{lasteq}
\phi_s^*\alpha = a \cdot \iota_s^\ast c_2(\mathcal{V}_{16}) + b' \cdot c_{Y_s}
\end{equation}
for some $a, b' \in \mathbb{Q}$ independent of $s \in \mathbb{P}_{16}$. Since $\iota_s^\ast\mathcal{V}_{16}$ is semi-rigid with Mukai vector $(3,H,5)$ by \cite[Proposition 2.2]{Muk5}, an identical argument as in the proof of Lemma \ref{lastLem} yields $a=0$ in the expression \eqref{lasteq}. 

This finishes the proof of Theorem \ref{Main} for $g=16$.


\begin{thebibliography}{10}

\bibitem{AC} E. Arbarello, M. Cornalba, {\em The Picard groups of the moduli spaces of curves.} Topology 26 (1987), no. 2, 153--171.

\bibitem{BV} A. Beauville, C. Voisin, {\em On the Chow ring of a $K3$ surface.} J. Algebraic Geom. 13 (2004), no. 3, 417--426.

\bibitem{EPS} G. Ellingsrud, R. Piene, S. A. Str{\o}mme, {\em On the variety of nets of quadrics defining twisted cubics.} Space curves (Rocca di Papa, 1985), 84--96, Lecture Notes in Math., 1266, Springer, Berlin, 1987.

\bibitem{ES} G. Ellingsrud, S. A. Str{\o}mme, {\em The number of twisted cubic curves on the general quintic threefold.} Math. Scand. 76 (1995), no. 1, 5--34.

\bibitem{Huy} D. Huybrechts, {\em Chow groups of $K3$ surfaces and spherical objects.} J. Eur. Math. Soc. (JEMS) 12 (2010), no. 6, 1533--1551.

\bibitem{HL} D. Huybrechts, M. Lehn, {\em The geometry of moduli spaces of sheaves. Second edition.} Cambridge Mathematical Library, Cambridge University Press, Cambridge, 2010, xviii+325 pp.


\bibitem{Mum} D. Mumford, {\em Rational equivalence of $0$-cycles on surfaces.} J. Math. Kyoto Univ. 9 (1968), 195--204.

\bibitem{Muk1} S. Mukai, {\em Curves, $K3$ surfaces and Fano $3$-folds of genus $\leq 10$.} Algebraic geometry and commutative algebra, Vol. I, 357--377, Kinokuniya, Tokyo, 1988.


\bibitem{Muk3} S. Mukai, {\em Polarized $K3$ surfaces of genus $18$ and $20$.} Complex projective geometry (Trieste, 1989/Bergen, 1989), 264--276, London Math. Soc. Lecture Note Ser., 179, Cambridge Univ. Press, Cambridge, 1992.

\bibitem{Muk4} S. Mukai, {\em Polarized $K3$ surfaces of genus thirteen.} Moduli spaces and arithmetic geometry, 315--326, Adv. Stud. Pure Math., 45, Math. Soc. Japan, Tokyo, 2006.

\bibitem{Muk5} S. Mukai, {\em $K3$ surfaces of genus sixteen.} Preprint, 2012, available at {\tt http://www.kurims. kyoto-u.ac.jp/preprint/file/RIMS1743.pdf}.

\bibitem{OG} K. G. O'Grady, {\em Moduli of sheaves and the Chow group of $K3$ surfaces.} J. Math. Pures Appl. (9) 100 (2013), no. 5, 701--718.

\bibitem{Ped} C. Pedrini, {\em Bloch's conjecture and valences of correspondences for $K3$ surfaces.} Preprint, 2015, {\tt arXiv:1510.05832v1}.

\bibitem{Ro} A. A. Rojtman, {\em The torsion of the group of $0$-cycles modulo rational equivalence.} Ann. of Math. (2) 111 (1980), no. 3, 553--569.

\bibitem{Ta} H. Tamvakis, {\em Quantum cohomology of isotropic Grassmannians.} Geometric methods in algebra and number theory, 311--338, Progr. Math., 235, Birkh\"auser Boston, Boston, MA, 2005.

\bibitem{V1} C. Voisin, {\em Chow rings, decomposition of the diagonal, and the topology of families.} Annals of Mathematics Studies, 187, Princeton University Press, Princeton, NJ, 2014, viii+163 pp.

\bibitem{V2} C. Voisin, {\em Rational equivalence of $0$-cycles on $K3$ surfaces and conjectures of Huybrechts and O'Grady.} Recent advances in algebraic geometry, 422--436, London Math. Soc. Lecture Note Ser., 417, Cambridge Univ. Press, Cambridge, 2015.

\end{thebibliography}
\end{document}